\documentclass{amsart}
\usepackage{shortsalch, xy,mathabx}
\xyoption{all}
\title[Denominators of special values count $\pi_*(L_{KU}S/\MakeLowercase{p})$.]{Denominators of special values of $\zeta$-functions count $KU$-local homotopy groups of mod $p$ Moore spectra.}
\author{A. Salch}
\usepackage{cleveref}
\begin{document}
\begin{abstract}
In this note, for each odd prime $p$, we show that the orders of the $KU$-local homotopy groups of the mod $p$ Moore spectrum are equal to denominators of special values of certain quotients of Dedekind zeta-functions of totally real number fields. With this observation in hand, we give a cute topological proof of the Leopoldt conjecture for those number fields, by showing that it is a consequence of periodicity properties of $KU$-local stable homotopy groups.
\end{abstract}
\maketitle
\tableofcontents

\section{Introduction.}

This note combines a calculation in stable homotopy theory from the 1970s with some number-theoretic ideas of Leopoldt and Carlitz from the 1950s. 
In~\cite{MR0198470}, J. F. Adams famously proved the following theorem:
\begin{theorem} {\bf (Adams.)} \label{adams thm} Let $n$ be a positive integer,
and let $\denom(\zeta(-n))$ be the denominator of the rational number $\zeta(-n)$ when written in reduced form. 
Then the image of the Whitehead $J$-homomorphism
\[ \mathbb{Z} \cong \pi_{4n+3}(SO) \stackrel{J}{\longrightarrow} \pi_{4n+3}^S(S^0)\]
is a cyclic group of order equal to the denominator of $\zeta(-2n-1)$, up to multiplication by a power of $2$.
\end{theorem}
Here $\zeta$ is the Riemann zeta-function, $\pi_{4n+3}(SO)$ is the $(4n+3)$rd (unstable) homotopy group of the infinite special orthogonal group, 
and $\pi_{4n+3}^S(S)$ is the $(4n+3)$rd stable homotopy group of the zero-sphere $S^0$.

Very closely related to the above result of Adams, one has Ravenel's computation (see~\cite{MR737778}, where it is mentioned that early versions of this computation were done by Adams and Baird):
\begin{theorem} {\bf (Ravenel.)}\label{adams-baird} Let $KU$ be periodic complex $K$-theory, and let $L_{KU}S$ be the Bousfield localization of the sphere spectrum $S$ at $KU$. 
Then, for all positive integers $n$,  
the order $\#(\pi_{2n}(L_{KU}S))$ of the $2n$th stable homotopy group $\pi_{2n}(L_{KU}S)$ is a power of $2$, and 
$\#(\pi_{2n-1}(L_{KU}S)) = \denom(\zeta(1-n))$ up to multiplication by a power of $2$.
\end{theorem}
In Theorem~\ref{adams-baird} we are adopting the convention that the denominator of the rational number $0$ is $1$. (This matters since $\zeta(-n) = 0$ for all even positive integers $n$.)

To date, Theorem~\ref{adams-baird} is unique in the literature, as the only description of the orders of the homotopy groups of a Bousfield-localized finite spectrum in terms of special values of $L$-functions.
The purpose of this note is to give an infinite collection of new examples of this phenomenon, by extending Theorem~\ref{adams-baird} to a family of spectra other than the sphere spectrum: we prove that the orders of the groups $\pi_*(L_{KU}S/p)$, where $S/p$ is the mod $p$ Moore spectrum (i.e., the homotopy cofiber of the degree $p$ map $S \rightarrow S$), are also denominators of special values of a natural $L$-function, when $p>2$.  The work involved is not difficult, once one pinpoints what the correct $L$-function should be; the hard part in this project was simply finding that $L$-function to begin with.

The main result in this note (proven in Proposition~\ref{euler product for L(S/p)}, Theorem~\ref{S/p realizability}, and Corollary~\ref{S/p realizability of norms}) is as follows:
\begin{theorem}\label{main thm summary}
Let $p$ be an odd prime. 
Then
for all positive integers $n$, if we agree to write $\denom(x)$ for the denominator of a rational number $x$ when written in reduced form,
we have equalities
\begin{equation}\label{main thm equality 0349} \denom\left(\frac{\zeta_F(1-n)}{\zeta(1-n)}\right) = \#(\pi_{2n}(L_{KU}S/p)) = \#(\pi_{2n-1}(L_{KU}S/p)) ,\end{equation}
where $L_{KU}S/p$ is the Bousfield localization of the mod $p$ Moore spectrum $S/p$ at periodic complex $K$-theory $KU$, and
where $F/\mathbb{Q}$ is the (unique) minimal subextension of $\mathbb{Q}(\zeta_{p^2})/\mathbb{Q}$ in which $p$ is wildly ramified.
\end{theorem}



In Theorem \ref{leopoldt conj special case}, we prove the Leopoldt conjecture, at the prime $p$, for each of the number fields $F$ described in the statement of Theorem \ref{main thm summary}. These number fields are all abelian, so the Leopoldt conjecture was already proven for them, by the work of Baker and Brumer \cite{MR220694}. The proof we offer for Theorem \ref{leopoldt conj special case} has the curious feature that it deduces the relevant cases of the Leopoldt conjecture from Colmez's $p$-adic class number formula, and Theorem \ref{main thm summary}, and $v_1$-periodicity in stable homotopy groups. One naturally wants to know if other, perhaps nonabelian, cases of the Leopoldt conjecture can be verified using a similar topological approach. We describe the general approach in Observation \ref{how to try to prove leopoldt conj}:
\begin{unnumberedobservation} 
If $F$ is a totally real number field and if we have integers $j,k$ and spectra $E,X$ such that
\begin{enumerate}
\item the order of $\pi_{2(p^k-1)p^n - 1}(L_EX)$ is equal to the denominator of the rational number $\frac{\zeta_F(1-p^{n}(p^k-1))}{\zeta(1-p^{n}(p^k-1))}$, 
\item $X$ admits a self-map $\Sigma^{(2p^k-2)j}X \rightarrow X$ which induces an isomorphism in $E_*$-homology, and
\item $\pi_{-1}(L_EX)$ is finite,
\end{enumerate}
then the Leopoldt conjecture holds for $F$ at the prime $p$.
\end{unnumberedobservation}
The periodicity theorem of Hopkins-Smith (see \cite{MR1652975}, or Theorem 1.5.4 of \cite{MR1192553}) gives an ample supply of spectra $X$ satisfying the second of the three conditions; see Remark \ref{using observation on leopoldt} for some discussion.

In the appendix, \cref{computed examples appendix}, we give some examples of computed special values (including numerators) of $\zeta_F(1-n)/\zeta(1-n)$, and an amusing relationship between the orders of homotopy groups of $L_{KU}S/p$ and the probability that certain ``random'' collections of integers satisfy certain coprimality conditions.

\begin{remark}
There are various known and conjectured relationships between orders of algebraic $K$-groups of number rings, and special values of Dedekind zeta-functions, such as Lichtenbaum's conjecture, from \cite{MR0406981}. Since the algebraic $K$-group $K_i(R)$ of a ring $R$ is the homotopy group $\pi_i(\mathcal{K}(R))$ of the algebraic $K$-theory spectrum $\mathcal{K}(R)$, one naturally wants to know how Theorem \ref{main thm summary} fits with algebraic $K$-theory. The answer is this: it follows from Thomason's identification (in \cite{MR826102}) of $p$-complete $\pi_n\left(\left(L_{KU}\mathcal{K}(R)\right)^{\widehat{}}_p\right)$ with $p$-complete \'{e}tale $K$-theory $K^{\et}_n(R)^{\widehat{}}_p$, for $n\gg 0$, together with Quillen's calculation of the $K$-groups of finite fields in \cite{MR0315016}, that the $p$-completion of $L_{KU}S$ is homotopy-equivalent to the $p$-completion of $L_{KU}\mathcal{K}(\mathbb{F}_{\ell})$ for any prime $\ell$ which is a primitive root modulo $p^2$. So, with some effort, one can rewrite the homotopy groups appearing in Ravenel's theorem reproduced above as Theorem \ref{adams-baird} as algebraic $K$-groups. But $L_{KU}S/p$ is not homotopy-equivalent to the algebraic $K$-theory spectrum of any finite field, or any number ring, or any number field, even after $p$-completion, even after restricting attention to homotopy groups in degrees $\gg 0$. So the results of this note do not seem to fit cleanly into any known or conjectured relationships between algebraic $K$-groups and special values.
\end{remark}

\begin{history-and-status-of-this-paper}
I wrote most of this material in 2016, but never publicly posted it, because I had the sense that there ought to be a more compelling, deeper, and more generalizable way to prove the same results. It took a few years for me to learn enough Iwasawa theory to find that {\em better} proof of these results, a proof that generalizes far beyond the mod $p$ Moore spectra, for example. 
But perhaps there is some value in making this note publicly available, since I think the ideas are quite interesting, and they are presented here in a way that doesn't require the reader to make an investment in learning Iwasawa theory, and because some versions of this note were privately circulated and I have been asked about it by several people. 
So I hope the reader will forgive me for presenting in this note only a precursor of what I think must be the really {\em effective} techniques for relating orders of stable homotopy groups to special values of $\zeta$-functions. 
Below, in Remark \ref{sorry everybody}, I sketch how to prove the main result of this note using those more effective (Iwasawa-theoretic) techniques.

The author wants to emphasize that the proofs in this document are all pretty easy; the hard work involved in this project was finding the correct function \linebreak$L(s,S/p) = \zeta_F(s)/\zeta(s)$ and fields $F$. 
The homotopy groups $\pi_*(L_{KU}S/p)$ for $p>2$ are very simple (namely, $\pi_n(L_{KU}S/p)$ is isomorphic to $\mathbb{Z}/p\mathbb{Z}$ if $n$ is congruent to $0$ or $-1$ modulo $2p-2$, and is trivial otherwise), but ``handcrafting'' an $L$-function to have rational special values with specified denominators at negative integers is a nontrivial task: ``most'' $L$-functions (in the usual families: Dedekind, Hasse-Weil, Artin...) with rational special values at negative integers typically {\em vanish} at negative integers, and of those which do not vanish, most are {\em integral} at negative integers, and of those which have nonzero noninteger rational special values, most seem to follow the same pattern of denominators as the Riemann zeta-function. Finding $\zeta_F(s)/\zeta(s)$, and the particular number fields $F$ described in Theorem \ref{main thm summary}, took the author some work; but once you have the right idea for $F$ and the idea to study $\zeta_F(s)/\zeta(s)$, the pieces fall into place using established methods.
\end{history-and-status-of-this-paper}

\begin{remark}\label{sorry everybody}
The argument we present for \eqref{main thm equality 0349} in this note is simply that one computes the denominators of $\denom\left(\frac{\zeta_F(1-n)}{\zeta(1-n)}\right)$, one compares it to the (already computed) order of $\pi_{2n}(L_{KU}S/p)$ and of $\pi_{2n-1}(L_{KU}S/p)$, and one sees that they are equal. So it is important to ask: {\em is equation \eqref{main thm equality 0349} just a coincidence?} I think the really compelling argument that it {\em isn't} a coincidence comes from an Iwasawa-theoretic proof of \eqref{main thm equality 0349}, which proceeds by {\em not} computing both sides of the equation \eqref{main thm equality 0349}, but instead by showing that the input for the descent spectral sequence (described below, in \eqref{descent ss 1}) computing $\pi_*(L_{KU}S/p)$ is the cohomology of a certain unit group Iwasawa module, whose cohomology groups are also (by the totally real case of the Iwasawa main conjecture, as in \cite{MR1053488}) the denominators of the special values of $\zeta_F(s)/\zeta(s)$ at negative integers; then the vanishing of the differentials in the spectral sequence gives equality \eqref{main thm equality 0349}. That Iwasawa-theoretic argument is beyond the scope of this note, and I will have to present it elsewhere. Since that Iwasawa-theoretic argument requires much more knowledge of algebraic number theory than the more classical, Dirichlet-character-theoretic approach in this note, I believe this note will be far more readable to an audience of topologists than a paper which describes the more general and powerful Iwasawa-theoretic approach.
\end{remark}

I have tried to make this note readable for number theorists, but I think this note will still be most accessible to a reader which is, like the author, trained in homotopy theory but not in number theory. A ``crash course'' in the necessary results from number theory can be found in \cref{review from num thy}, and a briefer crash course on the relevant topological results in \cref{topology review section}.


It is a pleasure to thank R. Bruner for many fruitful conversations relating to this material, D. Ravenel for support and inspiration in studying connections between special values and orders of homotopy groups, and an anonymous referee for helpful comments. The computer algebra packages MAGMA and SAGE were also indispensable in making large-scale systematic calculations of special values that led me, eventually, to zero in on the correct families of $L$-functions and finally the correct definition of $L(s,S/p)$.

\begin{conventions}
Throughout, we write $S$ for the sphere spectrum, $\zeta$ for the Riemann zeta-function, and $\nu_p(x)$ for the $p$-adic valuation of a number $x$.
\end{conventions}

\section{Review of $KU$-localization and the $KU$-local mod $p$ Moore spectrum.}
\label{topology review section}

This section explains some well-known ideas from stable homotopy theory which we will use. An excellent reference for this material is \cite{MR737778}.
\begin{definition}
Fix a spectrum $E$.
\begin{itemize}
\item We say that a map of spectra $f: X\rightarrow Y$ is an {\em $E$-local equivalence} if $E\smash f$ is a weak equivalence. In other words: $f$ is an $E$-local equivalence if and only if $f$ induces an isomorphism $E_*(X) \rightarrow E_*(Y)$.
\item We say that a spectrum $X$ is {\em $E$-acyclic} if $E\smash X$ is contractible.
\item We say that a spectrum $X$ is {\em $E$-local} if, for each $E$-acyclic spectrum $Y$, every map of spectra $Y \rightarrow X$ is nulhomotopic.
\item We say that a map of spectra $f: X \rightarrow L_EX$ is {\em the $E$-localization map on $X$,} and we call $L_EX$ the {\em Bousfield $E$-localization of $X$}, if $L_EX$ is $E$-local and $f$ is an $E$-local weak equivalence.
\end{itemize}
\end{definition}
Uniqueness (up to weak equivalence) of the Bousfield $E$-localization of $X$ is not difficult to see.
More difficult is the theorem of Bousfield that the Bousfield $E$-localization of $X$ exists, for all $E$ and all $X$.\footnote{The approach to Bousfield localization we have presented here is close to the original 1970s approach, as summarized in~\cite{MR737778}. There are later approaches as well: if we work with a model category of spectra which satisfies appropriate set-theoretic conditions, then there exists a ``coarser'' model structure on that same underlying category of spectra, whose cofibrations are the same, and whose ``coarse'' weak equivalences are precisely the $E$-local weak equivalences. Fibrant replacement in this ``coarse'' model structure is Bousfield $E$-localization. The book~\cite{MR1944041} is a very good reference for this elegant approach.}

Bousfield localization is, among other things, an analogue (for spectra) of the familiar notion of localization in commutative algebra: if $H\mathbb{Z}_{(p)}$ is the Eilenberg-Mac Lane spectrum of the $p$-local integers (uniquely determined by the property that $\pi_0(H\mathbb{Z}_{(p)})\cong \mathbb{Z}_{(p)}$ and $\pi_n(H\mathbb{Z}_{(p)})\cong 0$ for all $n\neq 0$), then $\pi_*\left(L_{H\mathbb{Z}_{(p)}}X\right) \cong \pi_*(X)_{(p)}$ for all spectra $X$ whose homotopy groups are bounded below, and $\pi_*\left(L_{L_{H\mathbb{Z}_{(p)}}S}X\right)  \cong \pi_*(X)_{(p)}$ for all spectra $X$ (without any bound required on homotopy groups). So {\em some} Bousfield localizations (like the ones just described, which simply $p$-localize the homotopy groups) have a predictable effect on the homotopy groups of spectra. 

However, Bousfield localization $L_E$ typically has a much more subtle effect on homotopy groups when $E$ is a spectrum which admits a homotopy equivalence $\Sigma^n E\stackrel{\cong}{\longrightarrow} E$ for some $n>0$. The effect of such Bousfield localizations on homotopy groups is at the core of the approach to stable homotopy groups of spheres via periodic phenomena in the chromatic tower and/or the Adams-Novikov spectral sequence; see \cite{MR1192553} for a survey. Let's consider the simplest case, the case where $E$ is $KU$, the periodic complex $K$-theory spectrum. Here is a very well-known and classical computation, dating back to at least the earlier circulated versions of \cite{MR737778}:
\begin{theorem}\label{classical computation}
Let $p$ be an odd prime, and let $S/p$ be the mod $p$ Moore spectrum. 
Then there is an isomorphism of graded abelian groups\footnote{This is also an isomorphism of graded rings, but we do not give a proof of that additional fact, because we do not work with multiplicative structure in this note.}
\[ \pi_*(L_{KU}S/p) \cong E(\alpha_1)\otimes_{\mathbb{F}_p} \mathbb{F}_p[v_1^{\pm 1}],\]
where $E(\alpha_1)$ is an exterior $\mathbb{F}_p$-algebra on a single generator $\alpha_1$ in degree $2p-3$, and $v_1$ is in degree $2p-2$. 

Consequently $\pi_n(L_{KU}S/p) \cong \mathbb{F}_p$ if $n$ is congruent to $0$ or $-1$ modulo $2p-2$, and $\pi_n(L_{KU}S/p) \cong 0$ otherwise.
\end{theorem}
\begin{proofsketch}
Here is one way (popularized by \cite{MR2030586}, where the ideas are generalized to formal groups of higher heights) to prove this result, which uses the spectral sequence
\begin{equation}\label{descent ss 1} H^*_c(\Aut(\mathbb{G}_1), E(\mathbb{G}_1)_*(X)) \Rightarrow \pi_*(L_{K(1)}X)\end{equation}
of~\cite{MR2030586}. (Here $H^*_c$ denotes profinite group cohomology, $\mathbb{G}_1$ is a formal group over $\mathbb{F}_p$ of height $1$, $E(\mathbb{G}_1)$ is its associated Morava/Lubin-Tate $E$-theory spectrum, and $K(1)$ is the first Morava $K$-theory at the prime $p$.)
It is classical (see e.g.~\cite{MR0172878}) that the profinite automorphism group $\Aut(\mathbb{G}_1)$ is isomorphic to the $p$-adic unit group $\hat{\mathbb{Z}}_p^{\times}$, and that (see e.g.~\cite{MR1333942}) $E(\mathbb{G}_1)_*$ is isomorphic to $\hat{\mathbb{Z}}_p[w^{\pm 1}]$ with $w$ in degree $-2$, with $\Aut(\mathbb{G}_1)$ acting on $\hat{\mathbb{Z}}_p\{ w^n\}$ by the $n$th power of the cyclotomic character, i.e., $u\cdot w^n$ is defined as the product $u^nw^n$.

Consequently $E(\mathbb{G}_1)_n(S/p)$ vanishes if $n$ is odd, and is isomorphic to $\mathbb{Z}/p\mathbb{Z}$ with $\Aut(\mathbb{G}_1) \cong \hat{\mathbb{Z}}_p^{\times}$ acting transitively if $n$ is even but not divisible by $2p-2$, and acting trivially if $n$ is divisible by $2p-2$. Easy Lyndon-Hochschild-Serre spectral sequence arguments then show that $H^*_c(\Aut(\mathbb{G}_1); E(\mathbb{G}_1)_n(S/p))$ vanishes unless $n$ is divisible by $2p-2$, and
\[ H^*_c(\Aut(\mathbb{G}_1); E(\mathbb{G}_1)_{n}(S/p)) \cong 
   H^*_c(1 + p\hat{\mathbb{Z}}_p; \mathbb{F}_p) \]
if $n$ is divisible by $2p-2$. Here $1 + p\hat{\mathbb{Z}}_p$ is the subgroup of $\hat{\mathbb{Z}}_p^{\times}$ consisting of units congruent to $1$ modulo $p$.

For $p>2$, convergence of the $p$-adic exponential map yields an isomorphism of profinite groups $p\hat{\mathbb{Z}}_p \stackrel{\cong}{\longrightarrow} 1 + p\hat{\mathbb{Z}}_p$, hence 
\[  H^j_c(1 + p\hat{\mathbb{Z}}_p; \mathbb{F}_p) 
 \cong \colim_m H^j_c(\mathbb{Z}/p^m\mathbb{Z}; \mathbb{F}_p) \]
is isomorphic to $\mathbb{F}_p$ if $j=0$ or $j=1$, and vanishes otherwise.
 
There is no room for differentials in spectral sequence~\eqref{descent ss 1}, so we get an isomorphism of graded abelian groups
\begin{equation*} \pi_*(L_{K(1)}S/p) \cong E(\alpha_1)\otimes_{\mathbb{F}_p} \mathbb{F}_p[v_1^{\pm 1}].\end{equation*}
with the degrees of $\alpha_1$ and $v_1$ as stated. 

Now since $S/p$ is already $S_{(p)}$-local, the $KU$-localization of $S/p$ coincides with the $KU_{(p)}$-localization of $S/p$. The well-known splitting
$KU_{(p)} \simeq \coprod_{j=0}^{p-2}\Sigma^{2j} E(1)$, where $E(1)$ is the $p$-local height $1$ Johnson-Wilson spectrum, establishes that $L_{KU_{(p)}}$
agrees with $L_{E(1)}$. The well-known homotopy pullback square\footnote{The existence of this homotopy pullback square seems to have been known since at least the 1980s, but as far as I know, there is no clear person or paper to whom the result is attributed. A nice modern writeup appears in Bauer's chapter \cite{bauer2011bousfield} in the book \cite{douglas2014topological}.}
\[\xymatrix{ L_{E(1)} X \ar[r]\ar[d] & L_{K(1)}X \ar[d] \\ L_{E(0)} X \ar[r] & L_{E(0)}L_{K(1)}X ,}\]
in the case $X = S/p$, then yields a weak equivalence $L_{E(1)}X \simeq L_{K(1)}X$, since $E(0)$-localization coincides with rationalization and so $L_{E(0)}S/p$ and $L_{E(0)}L_{K(1)}S/p$ are both contractible.
So 
\[ L_{KU}S/p \simeq L_{KU_{(p)}}S/p \simeq L_{E(1)}S/p \simeq L_{K(1)}S/p\]
has homotopy groups as stated.\hfill$\Box$
\end{proofsketch}

So the homotopy groups of $L_{KU}S/p$, for $p$ odd, are of a very simple form: $\pi_n(L_{KU}S/p)$ has order $p$ if $n$ is congruent to $0$ or $-1$ modulo $2p-2$, and $\pi_n(L_{KU}S/p)$ is trivial otherwise. To describe these groups in terms of special values of an $L$-function, as the work of Adams and Ravenel did (away from $2$) for $L_{KU}S$ as described in Theorem~\ref{adams-baird}, we need to find, for each odd prime, an $L$-function whose special values at negative integers are rational numbers whose denominators follow this same $(p-1)$-periodic pattern. We accomplish this in~\cref{l-function section}.

\section{Review of some ideas from number theory.}
\label{review from num thy}

\subsection{Review of Dirichlet characters and their $L$-functions.}

This section explains some well-known ideas from number theory which we will use. Excellent textbook references for this material include~\cite{MR0434929} and~\cite{MR1697859}.

The definition of a Dirichlet $L$-series and Dirichlet characters is classical:
\begin{definition}\label{def of l-series}
The {\em Dirichlet $L$-series} of a function $\chi: \mathbb{N} \rightarrow \mathbb{C}$ is the series
\begin{equation}\label{Lseries def} \sum_{n\geq 1}\frac{\chi(n)}{n^s}.\end{equation}
If $s$ is a complex number such that the series \eqref{Lseries def} converges, then we write $L(s,\chi)$ for the number that the series converges to\footnote{To be clear: for many functions $\chi$ of number-theoretic interest, the function $L(s,\chi)$ is meromorphic on some part of the complex plane, and admits a unique analytic continuation to a meromorphic function on a {\em larger} part of the complex plane. That analytic continuation is still called $L(s,\chi)$ for all $s$ in its domain, even though $L(s,\chi)$ only agrees with the series $\sum_{n\geq 1}\frac{\chi(n)}{n^s}$ for all complex $s$ such that the series $\sum_{n\geq 1}\frac{\chi(n)}{n^s}$ actually converges. For example: as we explain below, when $\chi$ is a Dirichlet character, the series $\sum_{n\geq 1}\frac{\chi(n)}{n^s}$ converges for all complex $s$ with $\Re(s)>1$, but it analytically continues to a meromorphic function on all of $\mathbb{C}$, and we write $L(-1,\chi)$ for that value of that meromorphic function at $s=-1$, even when the series $\sum_{n\geq 1}\frac{\chi(n)}{n^s}$ fails to converge when $s=-1$.}.

Given two functions $\chi_1,\chi_2: \mathbb{N}\rightarrow\mathbb{C}$, we define their {\em Dirichlet convolution} as the function $\chi_1\ast \chi_2: \mathbb{N}\rightarrow \mathbb{C}$ given by
\[ (\chi_1\ast \chi_2)(n) = \sum_{d\mid n} \chi_1(d)\chi_2(\frac{n}{d}),\]
so that $L(s,\chi_1\ast\chi_2) = L(s,\chi_1)L(s,\chi_2).$ (See Theorem~11.5 of~\cite{MR0434929} for a proof.)
\end{definition}

\begin{definition}\label{def of dirichlet l-series}
Let $f$ be a positive integer. A {\em Dirichlet character of modulus $f$} is a function $\chi: \mathbb{Z} \rightarrow \mathbb{C}$ satisfying the axioms:
\begin{itemize}
\item $\chi(1)=1$,
\item $\chi(n+f) = \chi(n)$ for all $n\in\mathbb{Z}$, 
\item $\chi(mn) = \chi(m)\chi(n)$ for all $m,n$, and
\item $\chi(n) = 0$ if $\gcd(n,f) \neq 1$.
\end{itemize}
A {\em Dirichlet character} is a Dirichlet character of modulus $f$ for some $f$.

The Dirichlet character $\chi_0$ of modulus $f$ such that $\chi_0(n) = 1$ for all $n$ coprime to $f$ is called the {\em principal} Dirichlet character of modulus $f$.

The {\em Dirichlet $L$-function} of a Dirichlet character $\chi$ is the Dirichlet $L$-series $\sum_{n\geq 1}\frac{\chi(n)}{n^s},$
which converges to a complex number $L(s,\chi)$ for all complex numbers $s$ with real part $>1$.
\end{definition}
For example, if $\chi_0$ is the (unique) character of modulus $1$, i.e., $\chi_0(n) = 1$ for all $n$, then $L(s,\chi_0) = \zeta(s)$, the Riemann zeta-function.

The Dirichlet characters of modulus $f$ form a group $\Dir(f)$ under pointwise multiplication; this group has order $\phi(f)$, and is cyclic if $f$ is a power of an odd prime. The Dirichlet characters of modulus $f$ do {\em not} form a group under Dirichlet convolution (see Definition~\ref{def of l-series}), since the Dirichlet convolution of two Dirichlet characters is not necessarily a Dirichlet character.

\begin{definition}
Let $\chi$ be a Dirichlet character of modulus $f$. A divisor $d$ of $f$ is called an {\em induced modulus for $\chi$} if $\chi(n) = 1$ for all $n$ relatively prime to $f$ such that $n \equiv 1$ modulo $d$.

A Dirichlet character $\chi$ of modulus $f$ is called {\em primitive} if the smallest induced modulus of $\chi$ is $f$ itself.
\end{definition}

Definition~\ref{def of gen berns} originally appeared in~\cite{MR0092812}.
\begin{definition} {\bf (Generalized Bernoulli numbers.)}\label{def of gen berns}
Let $\chi: \mathbb{Z}\rightarrow\mathbb{C}$ be a Dirichlet character of modulus $f$.
Let the sequence of numbers 
\[ B_1^\chi,B_2^\chi,B_3^\chi,\dots \in \mathbb{Q}(\zeta_{\phi(f)})\subseteq \mathbb{C}\] be defined as the Maclaurin coefficients
of $\sum_{r=1}^f \chi(r) \frac{te^{rt}}{e^{ft}-1}$, so that
\[ \sum_{r=1}^f \chi(r) \frac{te^{rt}}{e^{ft}-1} = \sum_{n\geq 0} B_n^\chi \frac{t^n}{n!} .\]
\end{definition}

The Euler product of $L(s,\chi)$ is classical:
\begin{equation}\label{euler product} L(s,\chi) = \prod_{\mbox{primes}\ p} \frac{1}{1 - \chi(p)p^{-s}}\end{equation}
for all complex numbers $s$ with $\Re(s)>1$.

See e.g. Theorem~VII.2.9 of~\cite{MR1697859} for Theorem~\ref{gen berns and dir l}:
\begin{theorem}  \label{gen berns and dir l}
The Dirichlet $L$-series $\sum_{n\geq 1} \frac{\chi(n)}{n^s}$
admits a unique analytic continuation to a meromorphic function $L(s,\chi)$ on the complex plane, and a functional equation such that
\[ L(1-n, \chi) = \frac{-B_n^\chi}{n}\]
for positive integers $n$.

Specifically, if $\chi$ is a primitive Dirichlet character of modulus $f$, then
\begin{equation}\label{funct eq} L(1 - s, \chi) = \frac{f^{s-1}\Gamma(s)}{(2\pi)^s}\left( e^{-\pi is/2} + \chi(-1)e^{\pi is/2}\right) G(1,\chi) L(s,\overline{\chi}),\end{equation}
where $\Gamma$ is the classical gamma-function (so $\Gamma(n) = (n-1)!$ for positive integers $n$), $\overline{\chi}$ is the complex-conjugate Dirichlet character of $\chi$ (so $\overline{\chi}(n) = \overline{\chi(n)}$), and
$G(1,\chi)$ is the Gauss sum $\sum_{r=1}^f \chi(r) e^{2\pi ir/f}$.
\end{theorem}
See e.g. Theorem~12.11 of~\cite{MR0434929} for equation~\eqref{funct eq}.
It is also classical that $L(s,\chi)$ is an entire function on the complex plane, if $\chi$ is nonprincipal; see e.g. Theorem 12.5 of \cite{MR0434929}.

\begin{observation}\label{vanishing of half the L-values}
If $\chi$ is a nonprincipal Dirichlet character with $\chi(-1) = 1$, then it is an easy exercise to show that the function $F(t) = \sum_{r=1}^f \chi(r) \frac{te^{rt}}{e^{ft}-1}$ satisfies $F(t) = F(-t)$, and hence that $B_n^{\chi} = 0$ for all odd positive integers $n$, hence that $L(1-n, \chi) = 0$ for all odd positive integers $n$.
\end{observation}

Since a Dirichlet character $\chi$ of modulus $f$ takes values in the $\phi(f)$th roots of unity, the numbers $B_n^\chi$ and $L(1-n,\chi)$ 
lie in the number field $\mathbb{Q}(\zeta_{\phi(f)})$.
Here is a theorem of Carlitz (see~\cite{MR0109132} for full proofs, or~\cite{MR0104630} for a shorter version) on how far these numbers are from being in the {\em ring of integers} of that number field:
\begin{theorem} \label{carlitz's thm}
Let $\chi$ be a primitive Dirichlet character of modulus $f$. 
\begin{itemize}
\item
If $f$ is not a prime power, then $\frac{B_n^\chi}{n}$ is an algebraic integer for all $n$.
\item 
If $f=p$ for some prime $p>2$, then let $g$ be a primitive root modulo $p^r$ for all $r$, i.e., $g\in\mathbb{N}$ is a topological generator of the group $\hat{\mathbb{Z}}_p^{\times}$ of $p$-adic units.
The number $\frac{B_n^\chi}{n}$ is an algebraic integer unless
$(p, 1 - \chi(g)) \neq (1)$, in which case $pB_n^{\chi} \equiv p-1$ modulo $(p, 1 - \chi(g))^{n+1}$. 
\item 
If $f=p^{\mu}$ for some prime $p>2$ and some integer $\mu>1$, then let $g$ be a primitive root modulo $p^r$ for all $r$, i.e., $g\in\mathbb{N}$ is a topological generator of the group $\hat{\mathbb{Z}}_p^{\times}$ of $p$-adic units.
The number $\frac{B_n^\chi}{n}$ is an algebraic integer unless
$(p, 1 - \chi(g)g^n) \neq (1)$, in which case $(1 -\chi(1+p))\frac{B_n^{\chi}}{n} \equiv 1$ modulo $(p, 1 - \chi(g)g^n)$. 
\item When $f$ is a power of a prime number $p$, then for all positive integers $n$,
$\frac{B_n^{\chi}}{n}\in \mathcal{O}_{\mathbb{Q}(\zeta_{\phi(f)})}[1/p]$.
That is, $\frac{B_n^{\chi}}{n}p^a$ is an algebraic integer for some (sufficiently large) positive integer $a$.
\end{itemize}
\end{theorem}

\subsection{Review of Dedekind zeta-functions.}

This material is classical; see e.g. chapter~3 of~\cite{MR1421575}.

\begin{definition}
Let $F/\mathbb{Q}$ be a finite field extension with ring of integers $\mathcal{O}_F$.
Then the {\em Dedekind zeta-function of $F$} is defined as the series
\[ \sum_{I\subseteq \mathcal{O}_F} \frac{1}{\left( \#\mathcal{O}_F/I\right)^{s}},\]
where the sum is taken over all nonzero ideals $I$ of $\mathcal{O}_F$, and 
$\#\mathcal{O}_F/I$ is the number of elements in the residue ring $\mathcal{O}_F/I$.
This series converges for complex numbers $s$ with real part $\Re(s) > 1$, and uniquely analytically continues to a meromorphic function $\zeta_F(s)$ on the complex plane.
\end{definition} 

The function $\zeta_F$ has the Euler product
\[ \zeta_F(s) = \prod_{\mathfrak{p}\subseteq \mathcal{O}_F} \frac{1}{1 - \#\left( \mathcal{O}_F/\mathfrak{p}\right)^{-s}}\]
for $\Re(s) > 1$,
where the product is taken over all nonzero prime ideals $\mathfrak{p}$ of $\mathcal{O}_F$.

\begin{definition}\label{def of assoc number field}
Let $f$ be a positive integer and let
$A$ be a subgroup of the group $\Dir(f)$ of Dirichlet characters of modulus $f$.
Let $\ker A$ denote the subgroup of $(\mathbb{Z}/f\mathbb{Z})^{\times}$ consisting of those residue classes $x$
such that $\chi(x) = 1$ for all $\chi\in A$.
Finally, let $G$ denote the subgroup of $\Gal(\mathbb{Q}(\zeta_f)/\mathbb{Q})$ corresponding to $\ker A \subseteq (\mathbb{Z}/f\mathbb{Z})^{\times}$ under the usual isomorphism $(\mathbb{Z}/f\mathbb{Z})^{\times} \stackrel{\cong}{\longrightarrow} \Gal(\mathbb{Q}(\zeta_f)/\mathbb{Q})$.
Then the {\em number field associated to $A$} is defined as the fixed field $\mathbb{Q}(\zeta_f)^{G}$.
\end{definition} 

Theorem~\ref{dedekind zeta and dirichlet L} combines Corollary 3.6 and Theorem 4.3 from \cite{MR1421575}.
\begin{theorem}\label{dedekind zeta and dirichlet L} 
Let $f$ be a positive integer, let
$A$ be a subgroup of the group $\Dir(f)$ of Dirichlet characters of modulus $f$, and let $F$ be the number field associated to $A$.
Then a prime $p\in\mathbb{Z}$ is unramified in $\mathcal{O}_F$ if and only if $\chi(p) \neq 0$ for all $\chi\in A$.
Furthermore:
\begin{equation}\label{washington formula} \zeta_F(s) = \prod_{\chi\in A} L(s, \chi).\end{equation}
\end{theorem}
Theorem~\ref{dedekind zeta and dirichlet L} requires a bit of care; the version of it expressed as Theorem~4.3 in~\cite{MR1421575} leaves one small (but important for getting correct Euler factors at ramified primes) point unexplained. The point is that the product should be taken over {\em primitive} representatives for Dirichlet characters in the group $A$. For the sake of the present work, what this means is the following: if $p>2$ and we let $A$ be the group $\Dir(p^2)[p]$ of $p$-torsion elements in the group of Dirichlet characters of modulus $p^2$, then every nonidentity element in the group $\Dir(p^2)[p]$ is a primitive Dirichlet character, but the identity element in $\Dir(p^2)[p]$ is the principal Dirichlet character of modulus $p^2$, which is imprimitive. Formula~\eqref{washington formula} is valid if, for the $L$-function factor corresponding to the identity element of $\Dir(p^2)[p]$, we use the Dirichlet $L$-function of the {\em primitive}, and consequently modulus $1$, representative for that identity element; i.e., we use the Riemann zeta-function. If we had instead used the Dirichlet $L$-series of the (imprimitive) principal Dirichlet character of modulus $p^2$, then formula~\eqref{washington formula} would be off by an Euler factor at $p$.

\section{The $L$-function of the mod $p$ Moore spectrum.}
\label{l-function section}

\begin{definition}\label{def of p-dir(p^2)}
Let $p$ be an odd prime. Since the group $\Dir(p^2)$ of Dirichlet characters of modulus $p^2$ is cyclic of order $\phi(p^2) = p(p-1)$, there exists a unique subgroup of index $p-1$ in $\Dir(p^2)$. We will write $\Dir(p^2)[p]$ for this subgroup.
Since $\Dir(p^2)[p]$ consists of the Dirichlet characters $\chi$ of modulus $p^2$  such that $\chi(n)$ is a $p$th root of unity for all $n$, the Galois group $G_{\mathbb{Q}(\zeta_p)/\mathbb{Q}} \cong C_{p-1}$ acts on $\Dir(p^2)[p]$. 
\end{definition}

\begin{definition}\label{def of S/p L-series}
Let $\chi$ be a generator of $\Dir(p^2)[p]$, and
let $L(s,S/p)$ denote the product
\[ L(s,S/p) = \prod_{\sigma\in G_{\mathbb{Q}(\zeta_p)/\mathbb{Q}}} L(s,\chi^{\sigma}).\]

Since the set of Dirichlet characters $\{ \chi^{\sigma}\}_{\sigma\in G_{\mathbb{Q}(\zeta_p)/\mathbb{Q}}}$ is exactly the set of nontrivial elements (i.e., nonprincipal characters) in $\Dir(p^2)[p]$, the function $L(s,S/p)$ is independent of the choice of $\chi$. 

Since each $L(s,\chi^{\sigma})$ converges for all complex numbers $s$ with real part $s>1$, the same is true of $L(s,S/p)$. Since each $L(s,\chi^{\sigma})$ has analytic continuation to a meromorphic function on the complex plane, the same is true of $L(s,S/p)$.
\end{definition}

\begin{observation}\label{easy observations about L(-,S/p)}
Here are some easy observations about $L(s,S/p)$:
\begin{itemize}
\item Since each $\chi^{\sigma}$ is nonprincipal, $L(s,\chi^{\sigma})$ is entire, so the product $L(s,S/p)$ is entire.
\item The function $L(s,S/p)$ can be written as a single $L$-series, as follows: let $\mathfrak{o}(p)$ denote the set of elements of order exactly $p$ in the group $\Dir(p^2)$ of Dirichlet characters of modulus $p^2$, i.e., $\mathfrak{o}(p)$ is the set of nonzero elements of $\Dir(p^2)[p]$. Let $\bigast_{\chi \in \mathfrak{o}(p)} \chi$ denote the Dirichlet convolution (see Definition~\ref{def of l-series}) of the elements in $\mathfrak{o}(p)$. 
Then \[ L(s, S/p) = L\left(s, \bigast_{\chi \in \mathfrak{o}(p)} \chi\right)  = \sum_{n\geq 1} \frac{\left(\bigast_{\chi \in \mathfrak{o}(p)} \chi\right)(n)}{n^s}.\]
\item It follows immediately from Observation~\ref{vanishing of half the L-values} that $L(1-n, S/p) = 0$ for all odd positive integers $n$.
\item We have the action of $G_{\mathbb{Q}(\zeta_p)/\mathbb{Q}}$ on $\Dir(p^2)[p]$ described in Definition~\ref{def of p-dir(p^2)}, and $G_{\mathbb{Q}(\zeta_p)/\mathbb{Q}}$ of course acts on $\mathbb{Q}(\zeta_p)$, where the elements of $\Dir(p^2)[p]$ take their values. Using Definition~\ref{def of gen berns}, one can compute $B_n^{\chi}$, for any fixed value of $n$, by solving for Taylor coefficients in a way which only involves a finite sum, and in particular, finitely many applications of $\chi$. So $\chi$ being a field automorphism implies $B_n^{(\chi^{\sigma})} = (B_n^{\chi})^{\sigma}$, 
for any $\sigma\in G_{\mathbb{Q}(\zeta_p)/\mathbb{Q}}$, without needing to know anything about continuity of the Galois action.
\item In particular, since $\mathbb{Q}(\zeta_p)/\mathbb{Q}$ is Galois, 
\begin{align*} 
 L(1-n,S/p) 
  &= \prod_{\sigma\in G_{\mathbb{Q}(\zeta_p)/\mathbb{Q}}}L(1-n,\chi^{\sigma}) \\
  &= \prod_{\sigma\in G_{\mathbb{Q}(\zeta_p)/\mathbb{Q}}}L(1-n,\chi)^{\sigma} \\
  &= N_{\mathbb{Q}(\zeta_p)/\mathbb{Q}}( L(1-n,\chi) ),\end{align*}
the field norm of the extension $\mathbb{Q}(\zeta_p)/\mathbb{Q}$, evaluated at $L(1-n, \chi)$. Consequently $L(1-n,S/p)\in\mathbb{Q}$.
\end{itemize}
\end{observation}

In Proposition~\ref{euler product for L(S/p)} we provide an Euler product formula for $L(s,S/p)$. The result is not number-theoretically novel at all: the method used is classical and very well-known.
The resulting formula 
involves a {\em division} by $\zeta(s)$, and consequently the numerators in special values of $\zeta(s)$ contribute (in an indirect way) to denominators in special values of $L(s,S/p)$. In Theorem~\ref{S/p realizability} and Corollary~\ref{S/p realizability of norms} we prove that the special values of $L(s,S/p)$ agree with the orders of the $KU$-local stable homotopy groups of the mod $p$ Moore spectrum $S/p$; consequently {\em numerators} of Bernoulli numbers are entering (again, in an indirect way) into the orders of $KU$-local stable homotopy groups. 
\begin{prop}\label{euler product for L(S/p)}
Let $p>2$ be a prime, and let $G_p$ be the set of prime numbers $\ell$ such that $\ell$  is a primitive root modulo $p^2$ (i.e., $\ell$ generates the group $(\mathbb{Z}/p^2\mathbb{Z})^{\times}$).
Let $N_p$ be the set of prime numbers $\ell \neq p$ not contained in $G_p$. 
Then, for all complex numbers $s$ with real part $>1$, we have an equality
\begin{equation}\label{L(S/p) euler product} L(s,S/p) =\frac{1-p^{-s}}{\zeta(s)} \left(\prod_{\ell\in N_p} \frac{1}{1 - \ell^{-s}}\right)^{p}\left(\prod_{\ell\in G_p} \frac{1}{1 - \ell^{-sp}}\right).\end{equation}
\end{prop}
\begin{proof}
The product
of the Euler products \eqref{euler product} of the $L$-functions $L(s,\chi^{\sigma})$ over all $\sigma\in G_{\mathbb{Q}(\zeta_p)/\mathbb{Q}}$,
together with the Euler product of $\zeta(s) = L(s,\chi_0)$,
has three types of factors:
\begin{itemize}
\item the $L$-factor at $\ell$, for primes $\ell\in G_p$, is
\[ \zeta_{\ell}(s)L_{\ell}(s,S/p) = \frac{1}{1-\ell^{-s}}\left(\prod_{\sigma\in G_{\mathbb{Q}(\zeta_p)/\mathbb{Q}}} (1-\chi^{\sigma}(\ell)\ell^{-s})\right)^{-1}.\]
Since $\ell\in G_p$, the prime $\ell$ is a primitive root modulo $p^2$, hence $\chi(\ell)$ is a primitive $p$th root of unity.
Hence 
\[\frac{\ell^{ps}}{\zeta_{\ell}(s)L_{\ell}(s,S/p)} = \prod_{j=1}^{p}(\ell^{s}-\zeta_p^j) = (1 - \ell^s)\Phi_p(\ell^s) = 1 - \ell^{ps},\]
with $j$ a primitive $p$th root of unity, and where $\Phi_p$ is the $p$th cyclotomic polynomial.
Solving for $\zeta_{\ell}(s)L_{\ell}(s,S/p)$ yields $\frac{1}{1 - \ell^{-ps}}$.
\item For the $L$-factor at $\ell$, for primes $\ell\in N_p$,
we first observe that $\chi = \chi_{p^2}^{p-1}$ for an appropriately chosen primitive character $\chi_{p^2}$ of modulus $p^2$. Hence $\chi(n) = \chi_{p^2}(n)^{p-1}$ is a $p$th root of unity which is primitive if and only if $n$ is a primitive root modulo $p^2$. 
Consequently, if $\ell\in N_p$, then $\chi(\ell)$ is a non-primitive $p$th root of unity, i.e., $\chi(\ell) = 1$.
This lets us simplify an Euler factor:
\begin{align*} \zeta_{\ell}(s)L_{\ell}(s,S/p) 
 &= (1-\ell^{-s})^{-1}\left(\prod_{\sigma\in G_{\mathbb{Q}(\zeta_p)/\mathbb{Q}}} (1-\chi^{\sigma}(\ell)\ell^{-s})\right)^{-1}\\
 &= \left( 1-\ell^{-s}\right)^{-p}.\end{align*}
\item The third type of $L$-factor is simply the $\ell=p$ $L$-factor. Since $\chi(p) = 0$, the $p$-local $L$-factor in $\zeta(s)L(s,S/p)$ is simply the $p$-local $L$-factor in $\zeta(s)$, i.e., $(1-p^{-s})^{-1}$.
\end{itemize}
Taking a product over all primes $\ell$ yields the formula~\eqref{L(S/p) euler product}.
\end{proof}

\begin{lemma}\label{identification of fixed field}
The number field associated to the group $\Dir(p^2)[p]$ is the unique minimal 
subextension of $\mathbb{Q}(\zeta_{p^2})/\mathbb{Q}$ in which $p$ is wildly ramified. 
(See Definition~\ref{def of p-dir(p^2)} for the definition of $\Dir(p^2)[p]$, and see Definition~\ref{def of assoc number field} for the definition of the number field associated to a group of Dirichlet characters.)
\end{lemma}
\begin{proof}
Elementary exercise.
\end{proof}

\begin{prop}\label{L(-,S/p) is relative dedekind zeta}
Let $p$ be an odd prime, and let $F/\mathbb{Q}$ be the minimal subextension 
of $\mathbb{Q}(\zeta_{p^2})/\mathbb{Q}$ in which $p$ ramifies wildly. 
Then:
\begin{equation}\label{relative dedekind zeta formula} \frac{\zeta_F(s)}{\zeta(s)} = L(s,S/p).\end{equation}
\end{prop}
\begin{proof}
Among the $p$ elements of the group $\Dir(p^2)[p]$, there are $p-1$ primitive Dirichlet characters, i.e., there are $p-1$ nonprincipal $\mathbb{Q}(\zeta_p)$-valued Dirichlet characters of modulus $p^2$. 
In the group $\Dir(p^2)[p]$ we also have the one imprimitive character, namely, the principal Dirichlet character of modulus $p^2$, which is the identity element of the group $\Dir(p^2)[p]$. Lemma \ref{identification of fixed field} and Theorem \ref{dedekind zeta and dirichlet L} together let us express $\zeta_F(s)$ as a product over {\em primitive representatives} of the elements of $\Dir(p^2)[p]$. So let us write $\Dir(p^2)[p]^{\prime}$ for the set of primitive representatives for the elements of $\Dir(p^2)[p]$, and let us write $\chi_p$ for any generator for the group $\Dir(p^2)[p]$. The Galois group $G_{\mathbb{Q}(\zeta_p)/\mathbb{Q}}$ acts freely on the $\mathbb{Q}(\zeta_p)$-valued primitive Dirichlet characters of modulus $p^2$, so that 
\[ \Dir(p^2)[p]^{\prime} = \{ \chi_0 \} \cup \{ \chi_p^{\sigma}: \sigma \in G_{\mathbb{Q}(\zeta_p)/\mathbb{Q}}\},\]
where $\chi_0$ is principal and modulus $1$.
Now we have equalities:
\begin{align}
\label{eq b1} \frac{\zeta_F(s)}{\zeta(s)} 
  &= \frac{\prod_{\chi\in \Dir(p^2)[p]^{\prime}} L(s,\chi)}{\zeta(s)} \\
\nonumber  &= \frac{L(s,\chi_0)\cdot \prod_{\sigma \in G_{\mathbb{Q}(\zeta_p)/\mathbb{Q}}} L(s,\chi_p^{\sigma})}{\zeta(s)} \\
\label{eq b2}  &= \frac{L(s,\chi_0)\cdot L(s,S/p)}{\zeta(s)} \\
\label{eq b3}  &= L(s,S/p),
\end{align}
where \eqref{eq b1} is from Lemma \ref{identification of fixed field} and Theorem \ref{dedekind zeta and dirichlet L}, and \eqref{eq b2} is from Definition \ref{def of S/p L-series}. Of course \eqref{eq b3} comes simply from the observation that the Riemann zeta-function is the Dirichlet $L$-function of the principal Dirichlet character of modulus $1$.
\end{proof}

\begin{lemma}\label{easy ideal containment lemma}
Let $p>2$ be a prime, let $g\in\mathbb{N}$ be a primitive root modulo $p^2$, let $n$ be a positive integer divisible by $p-1$, 
let $\zeta_p$ denote a primitive $p$th root of unity, and let $\mathcal{O}_{\mathbb{Q}(\zeta_p)}$ denote the ring of integers in $\mathbb{Q}(\zeta_p)$.
Then the ideal $(1 - \zeta_p g^n)$ in $\mathcal{O}_{\mathbb{Q}(\zeta_p)}$ is contained in the ideal $(1 - \zeta_p)$.
\end{lemma}
\begin{proof}
We will use the well-known factorization $(p) = (1-\zeta_p)^{p-1}$ in the ring of integers $\mathcal{O}_{\mathbb{Q}(\zeta_p)}$ (see e.g. Lemma~10.1 of~\cite{MR1697859}),
and its corollary, that $(1 - \zeta_p)$ is a maximal ideal in $\mathcal{O}_{\mathbb{Q}(\zeta_p)}$ with residue field $\mathbb{F}_p$.
Obviously $1 - \zeta_p g^n$ is congruent to $1 - g^n$ modulo $1 - \zeta_p$. Since $g$ is nonzero modulo $p$, we have $g^{p-1} \equiv 1$ in the residue field 
$\mathcal{O}_{\mathbb{Q}(\zeta_p)}/(1 - \zeta_p) \cong \mathbb{F}_p$. So $g^n \equiv 1$ modulo $(1 - \zeta_p)$, so $1 - \zeta_pg^n \equiv 0$ modulo $1 - \zeta_p$.
So $(1 - \zeta_pg^n) \subseteq (1 - \zeta_p)$.
\end{proof}

In Theorem~\ref{S/p realizability}, we adopt the convention that, given a rational number $x$, we write $\denom(x)$ for the denominator of $x$ when written in reduced form, and we let $\denom(x)$ be $1$ if $x=0$.
\begin{theorem}\label{S/p realizability} 
Let $p>2$ be a prime. 
Then, for each positive integer $n$, the following four numbers are equal:
\begin{itemize}
\item the order of $\pi_{2n}(L_{KU}S/p)$, the $(2n)$th stable homotopy group of the $KU$-local mod $p$ Moore spectrum,
\item the order of $\pi_{2n-1}(L_{KU}S/p)$, the $(2n-1)$th stable homotopy group of the $KU$-local mod $p$ Moore spectrum, 
\item $\denom(L(1-n,S/p))$, and
\item $\denom\left(\frac{\zeta_F(1-n)}{\zeta(1-n)}\right)$, where $F/\mathbb{Q}$ is the (unique) smallest subextension of \linebreak $\mathbb{Q}(\zeta_{p^2})/\mathbb{Q}$ in which $p$ is wildly ramified.
\end{itemize}
\end{theorem}
\begin{proof}
By Theorem~\ref{classical computation}, $\pi_{2n}(L_{KU}S/p)\cong \pi_{2n-1}(L_{KU}S/p)\cong \mathbb{Z}/p\mathbb{Z}$ if $(p-1)\mid n$,
and $\pi_{2n}(L_{KU}S/p)$ and $\pi_{2n-1}(L_{KU}S/p)$ are trivial if $(p-1)\nmid n$, 
so there are just two cases to consider:
\begin{description}
\item[If $p-1\mid n$]
The Dirichlet character $\chi$ of Definition~\ref{def of S/p L-series} 
has the property that $\chi(g)$ is a primitive $p$th root of unity, for any primitive root $g$ modulo $p^2$. 
So by Lemma~\ref{easy ideal containment lemma}, $(1 - \chi(g)g^n) \subseteq (1 - \zeta_p)$, 
so $(p, 1 - \chi(g)g^n) \subseteq (1 - \zeta_p)$ since $(1 - \zeta_p)^{p-1} = (p)$.
In particular, $(p, 1 - \chi(g)g^n)$ is contained in a maximal ideal of $\mathcal{O}_{\mathbb{Q}(\zeta_p)}$, so $(p, 1 - \chi(g)g^n) \neq (1)$.
(For this manipulation of ideals, it does not really matter which primitive $p$th root of unity $\zeta_p$ we choose, or whether or not it is actually equal to $\chi(g)$: there is a unique maximal ideal of $\mathcal{O}_{\mathbb{Q}(\zeta_p)}$ over $p$, and it is of the form $(1 - \zeta)$ for any primitive $p$th root of unity $\zeta$ we choose. So $(1 - \zeta_p) = (1-\zeta)$ for any primitive $p$th root of unity $\zeta$.)

Now we invoke Carlitz's result, Theorem~\ref{carlitz's thm}: 
since $(p, 1 - \chi(g)g^n) \neq (1)$, we have that
$(1 -\chi(1+p))\frac{B_n^{\chi}}{n}$ is congruent to $1$ modulo $(p, 1 - \chi(g)g^n)$. 
Since $(p, 1 - \chi(g)g^n) \subseteq (1 - \zeta_p)$, we now have 
\begin{align}
\nonumber (1 - \chi(1+p))^{p-1}L(1-n,S/p) 
  &= (1 - \chi(1+p))^{p-1} \prod_{\sigma\in G_{\mathbb{Q}(\zeta_p)/\mathbb{Q}}} \frac{-B_n^{\chi^{\sigma}}}{n} \\
\label{pre-valuation equality}  &\equiv 1\mod (1 - \zeta_p),
\end{align}
and, on taking $p$-adic valuations of both sides of the equation~\eqref{pre-valuation equality} in $\mathcal{O}_{\mathbb{Q}(\zeta_p)}$, we have
\begin{equation}\label{valuation equality} (p-1)\nu_p\left( 1 - \chi(1+p)\right) + \nu_p(L(1-n, S/p)) = 0.\end{equation}

Now remember that $\chi$ is a generator of $\Dir(p^2)[p]$, and in particular, $\chi$ takes primitive $p$th roots of unity in $\mathbb{Z}/p^2\mathbb{Z}$ to primitive $p$th roots of unity. 
By a very easy elementary computation, $1+p$ is a primitive $p$th root of unity in $\mathbb{Z}/p^2\mathbb{Z}$; and since we have an equality of ideals $(1 - \zeta) = (1 - \zeta_p)$ for any primitive $p$th root of unity $\zeta$ in $\mathcal{O}_{\mathbb{Q}(\zeta_p)}$, we now have
$\nu_p(1 - \chi(1+p)) = \nu_p(1 - \zeta_p) = \frac{1}{p-1}$. 
Hence equation~\eqref{valuation equality} gives us that $\nu_p(L(1-n,S/p)) = -1$, i.e., the denominator of the rational number $L(1-n,S/p)$ is divisible by $p$, but not by $p^2$.

Carlitz's result, Theorem~\ref{carlitz's thm}, also implies that $\frac{p^aB_n^{\chi}}{n}\in \mathcal{O}_{\mathbb{Q}(\zeta_p)}$ for some sufficiently large integer $a$.
So the product $L(1-n,S/p) = \prod_{\sigma\in G_{\mathbb{Q}(\zeta_p)/\mathbb{Q}}}\frac{-B_n^{\chi}}{n}$ has denominator divisible by no primes other than the prime $p$.

Consequently the denominator of $L(1-n,S/p)$ is exactly $p$, when $p-1$ divides $n$.
\item[If $p-1\nmid n$] Let $g\in\mathbb{N}$ be a primitive root modulo $p^2$. 
We have the congruence $1 - \chi(g)g^n \equiv 1 - g^n$ modulo $1 - \zeta_p$, for the same reasons as in the previous part of this proof.

The difference is now that, since $p-1$ does not divide $n$, 
$g^n \nequiv 1$ modulo $p$. So $1 - g^n \nequiv 0$ modulo $(1 - \zeta_p) \subseteq (p)$, so 
$(p, 1 - \chi(g)g^n) = (1)\subseteq \mathcal{O}_{\mathbb{Q}(\zeta_p)}$. 
Now Theorem~\ref{carlitz's thm} implies that $\frac{B_n^\chi}{n}$ is an algebraic integer.
Hence, taking a product over Galois conjugates, we have
\[ L(1-n, S/p) = \prod_{\sigma\in G_{\mathbb{Q}(\zeta_p)/\mathbb{Q}}}\frac{-B_n^{\chi^{\sigma}}}{n} \in \mathbb{Z},\]
as desired.
\end{description}

Finally, the equality $\denom(L(1-n,S/p)) = \denom\left(\frac{\zeta_F(1-n)}{\zeta(1-n)}\right)$ is immediate from Proposition~\ref{L(-,S/p) is relative dedekind zeta}.
\end{proof}
Theorem~3 in~\cite{MR0392863} is very similar to Carlitz's theorem reproduced as Theorem~\ref{carlitz's thm}, above. Consequently the role of of Carlitz's theorem in the proof of Theorem~\ref{S/p realizability} can also be filled by part~2 of Theorem~3 in Fresnel's paper.

Corollary \ref{S/p realizability of norms} ought to be understood as a distant descendant of the classical theorem that $\frac{\zeta(n)}{\pi^n}$ is rational for positive even integers $n$.
\begin{corollary}\label{S/p realizability of norms}
Let $p>2$ be a prime,
and let $n$ be a positive even integer.
Then $L(s,S/p)$ satisfies the functional equation
\begin{equation*} L(n,S/p) = \left( \frac{2^{n-1}\pi^{n}}{p^{2n-1}(n-1)!}\right)^{p-1} L(1-n, S/p),  \end{equation*} up to sign.
Consequently the number 
$L(n,S/p)$ is equal to 
$\left( \frac{2^{n-1}\pi^{n}}{p^{2n-1}(n-1)!}\right)^{p-1}$ times a rational number which, when written in reduced form, has denominator
equal to the order of $\pi_{2n}(L_{KU}S/p)$ and of $\pi_{2n-1}(L_{KU}S/p)$ , the $(2n)$th and $(2n-1)$st $KU$-local stable homotopy groups of the mod $p$ Moore spectrum.
\end{corollary}
\begin{proof}
Let $\chi$ be as in Definition~\ref{def of S/p L-series},
i.e., $\chi$ is any primitive Dirichlet character of modulus $p^2$ taking values in $\mathbb{Q}(\zeta_p)$.
The function $L(1-n, \chi)$ vanishes when $n$ is an odd positive integer (by Observation~\ref{vanishing of half the L-values}); so suppose instead that
$n$ is an even positive integer.
Taking the complex norm of both sides of the functional equation~\eqref{funct eq},
we have that:
\begin{itemize}
\item $\left| G(1,\chi) \right| = p$ (since, for $\chi$ any primitive Dirichlet character of modulus $f$, we have $\left| G(1,\chi) \right| = \sqrt{f}$; see e.g. Theorem~8.15 of~\cite{MR0434929}),
\item since $\chi(-1) = 1$ and since $n$ is an even positive integer, we have
\[ \left| e^{-\pi i n/2} + \chi(-1) e^{\pi in/2}\right| = 2.\]
\end{itemize}
Consequently functional equation~\eqref{funct eq}
yields
\begin{align} \label{norm of funct eq} \left| L(1-n, \chi) \right|
 &= \frac{p^{2n-1}(n-1)!}{2^{n-1}\pi^n}\left| L(s,\overline{\chi})\right| \end{align}
for all positive even integers $n$.

In the group $\Dir(p^2)[p]$ of Dirichlet characters of modulus $p^2$ taking values in $\mathbb{Q}(\zeta_p)$ (see Definition~\ref{def of p-dir(p^2)}), complex conjugation acts freely on the nonprincipal characters; consequently, taking a product of the equation~\eqref{norm of funct eq} over all nonprincipal $\chi\in \Dir(p^2)[p]$, we get
\begin{equation}\label{solved norm funct eq} \left( \frac{2^{n-1}}{p^{2n-1}(n-1)!}\right)^{p-1} \left| L(1-n, S/p)\right| = \frac{1}{\pi^{n(p-1)}}\left| L(n,S/p)\right| \end{equation}
for all positive even integers $n$.

Now it follows from the Euler product for $L(s,S/p)$ (see Proposition~\ref{euler product for L(S/p)}) that $L(n,S/p)$ is a {\em real} number,
and it follows from the definition of $L(s,S/p)$ as a product of Galois conjugates of Dirichlet $L$-functions that $L(1-n,S/p)$, for positive integers $n$, is a product of Galois conjugates of generalized Bernoulli numbers, hence is rational.
Consequently equation~\eqref{solved norm funct eq} now reads
\begin{equation*} L(n,S/p) = \pm \left( \frac{2^{n-1}\pi^{n}}{p^{2n-1}(n-1)!}\right)^{p-1} L(1-n, S/p)  \end{equation*}
for all positive even integers $n$.
Now the description of the denominator of $L(1-n, S/p)$ given in Theorem~\ref{S/p realizability} implies the claimed result.
\end{proof}

\section{Consequences for the Leopoldt conjecture.}

For a totally real number field $F$, the classical class number formula\footnote{The form we give here is somewhat simpler than a typical textbook statement of the class number formula, since we give the formula only for totally real $F$. For example, for totally real $F$, the discriminant $\Delta_F$ is always positive, so there is no need to take the absolute value of $\Delta_F$ before taking its square root.} reads:
\begin{align}\label{classical class num formula} \lim_{s\rightarrow 1}(s-1)\zeta_F(s) &= \frac{2^{[F:\mathbb{Q}]} \reg_F h_F}{w_F \sqrt{\Delta_F}},\end{align}
where $\reg_F$ is the classical regulator of $F$, $h_F$ the class number of $F$, $w_F$ the number of roots of unity in $F$, and $\Delta_F$ the discriminant of $F$.
In \cite{MR922806}, Colmez proved an analogue of \eqref{classical class num formula} for the $p$-adic Dedekind $\zeta$-function of $F$:
\begin{theorem} {\bf (Colmez.)}
\begin{equation}\label{colmez eq 13}  \lim_{s\rightarrow 1}(s-1)\zeta_{F,p}(s) = \frac{2^{[F:\mathbb{Q}]} \reg_{F,p} h_F \prod_{\mathfrak{p}\mid p} \left(1-\frac{1}{N(\mathfrak{p})}\right)}{w_F\sqrt{\Delta_F}},\end{equation}
\end{theorem}
where $h_F,w_F,$ and $\Delta_F$ are the same as in \eqref{classical class num formula}, $N(\mathfrak{p})$ is the norm of the prime ideal $\mathfrak{p}$, and $\reg_{F,p}$ is Leopoldt's {\em $p$-adic regulator of $F$}, whose definition we give somewhat informally as follows:

\begin{definition}\label{def of p-adic reg}
Fix a prime number $p$, and\footnote{This embedding is used only so that, given an embedding $F\hookrightarrow \mathbb{C}_p$, we can say it's ``real'' or ``complex.''} an embedding $\mathbb{C}_p \hookrightarrow \mathbb{C}$.
Let $\sigma_1, \dots ,\sigma_r$ be the embeddings of $F$ into $\mathbb{C}_p$ (only list the complex embeddings ``once''---leave their conjugates off the list). Let $e_1, \dots ,e_{s}$ be a $\mathbb{Z}$-linear basis for $\mathcal{O}_F$; it follows from the Dirichlet unit theorem that $s=r-1$. 
The {\em $p$-adic regulator of $F$} is
\begin{equation}\label{def of p-adic reg 0} \reg_{F,p} = \det( \delta_i \log_p(\sigma_i(e_j)))_{1\leq i,j\leq s},\end{equation}
where $\delta_i$ is $1$ if $\sigma_i$ is real and $2$ if $\sigma_j$ is complex, and where $\log_p$ is the $p$-adic logarithm (take the Maclaurin series for $\ln (1+x)$, but regard the coefficients as $p$-adic rationals).
\end{definition}
In Definition \ref{def of p-adic reg}, we see that we naturally get an $s$-by-$(s+1)$ matrix of $p$-adic logarithms of the numbers $\sigma_i(e_j)$, and in \eqref{def of p-adic reg 0} we simply ignore one of the columns to get a square matrix, whose determinant we define as the $p$-adic regulator; omitting a different column swaps the sign of the determinant, and so $\reg_{F,p}$ is only well-defined up to sign. See \cite{MR139602} for further explanation.

The Leopoldt conjecture is simply the conjecture that $\reg_{F,p}$ is nonzero for all primes $p$ and all number fields $F$.

Siegel \cite{MR1503335} and Klingen \cite{MR133304} proved that $\zeta_F(1-n)$ is an algebraic rational number when $n$ is a positive integer, so we can think of the sequence
\begin{equation}\label{seq 0493409}\zeta_F(0),\zeta_F(-1),\zeta_F(-2),\dots\end{equation} as a sequence of {\em $p$-adic} numbers, and we can ask whether there exists some continuous function on the $p$-adic integers whose values at negative integers recovers the sequence \eqref{seq 0493409}. This doesn't work, but for only two reasons, and each can be dealt with by modifying the question appropriately: what {\em does} work is to cancel out the Euler factors in $\zeta_F(s)$ corresponding to primes (of the ring of integers $\mathcal{O}_F$) over $p$, and then to evaluate the result at $s=1-(p-1), 1-2(p-1), 1-3(p-1), 1-4(p-1),\dots$ instead of at $s=0,-1,-2,\dots$.

One then arrives at the result, from \cite{MR0404145}, that there exists a unique $p$-adically continuous function $\zeta_{F,p}(s)$ 
 such that, for $n>1$ an integer, 
\begin{align}\label{eq 010} \zeta_{F,p}(1-n(p-1)) 
 &= \zeta_F(1-n(p-1))\prod_{\mathfrak{p}\mid p}\left(1 - N(\mathfrak{p})^{n(p-1)-1}\right)\end{align}
if $p>2$, and 
\begin{align}\label{eq 011} \zeta_{F,p}(1-2n) 
 &= \zeta_F(1-2n)\prod_{\mathfrak{p}\mid p}\left(1 - N(\mathfrak{p})^{2n-1}\right)\end{align}
if $p=2$; this result extends the results of \cite{MR163900}, which assumed $F$ abelian.
In particular, $\nu_p\left(\zeta_F(1-n(p-1))\right) = \nu_p\left( \zeta_{F,p}(1-n(p-1))\right)$ and $\nu_2\left(\zeta_F(1-2n)\right) = \nu_2\left( \zeta_{F,2}(1-2n)\right)$ for positive integers $n$.

Since $\zeta_{F,p}$ is $p$-adically continuous and since the sequence of integers
\[ \left( 1-(p-1), 1-p(p-1), 1-p^2(p-1), 1-p^3(p-1), \dots\right)\]
converges $p$-adically to $1$, using \eqref{colmez eq 13} we get an equality:
\begin{align}
\label{cor of colmez 230} \frac{2^{[F:\mathbb{Q}]} \reg_{F,p} h_F \prod_{\mathfrak{p}\mid p} \left(1-\frac{1}{N(\mathfrak{p})}\right)}{w_F\sqrt{\Delta_F}}
  &= \lim_{j\rightarrow \infty} (-p^j(p-1)) \zeta_{F,p}\left( 1-p^j(p-1)\right) .
\end{align}
The trick now is to compare \eqref{cor of colmez 230} for a nontrivial choice of $F$ to \eqref{cor of colmez 230} for the trivial choice of $F$, i.e., $F=\mathbb{Q}$, and to compare the resulting ratio to the order of a homotopy group using Theorem \ref{S/p realizability}. 
Suppose now that $F$ is the smallest subextension of $\mathbb{Q}(\zeta_{p^2})/\mathbb{Q}$ in which $p$ is wildly ramified.
Then we have:
\begin{align*}
 p &= \# \left( \pi_{2(p-1)p^n - 1}( L_{KU}S/p)\right) \\
   &= \denom\left( \frac{\zeta_{F}\left( 1 - p^n(p-1)\right)}{\zeta\left(1 - p^n(p-1)\right)}\right), \\
   &= \denom\left( \frac{\zeta_{F,p}\left( 1 - p^n(p-1)\right)}{\zeta_{\mathbb{Q},p}\left(1 - p^n(p-1)\right)} \right),
\end{align*}
so the order of vanishing of $\zeta_{K,p}(s)$ at $s=1$ is equal to the order of vanishing of $\zeta_{\mathbb{Q},p}(s)$ at $s=1$. We have that
\begin{align*}
 \left(1-\frac{1}{p}\right)\reg_{\mathbb{Q},p} &= \lim_{s\rightarrow 1}(s-1)\zeta_{\mathbb{Q},p}(s),
\end{align*}
which is nonzero, so $\lim_{s\rightarrow 1}(s-1) \zeta_{F,p}(s)$ also converges and is nonzero.
Colmez's class number formula \eqref{colmez eq 13} then yields that
\begin{equation}\label{cnf product} \frac{2^{[F:\mathbb{Q}]} \reg_{F,p} h_F \prod_{\mathfrak{p}\mid p} \left(1-\frac{1}{N(\mathfrak{p})}\right)}{w_F\sqrt{\Delta_F}}\end{equation}
must be nonzero. Each factor in \eqref{cnf product} is automatically nonzero, except possibly for $\reg_{F,p}$; so the $p$-adic regulator $\reg_{F,p}$ of $F$ must also be nonzero, i.e., 
\begin{theorem}\label{leopoldt conj special case}
Let $F$ be the smallest subextension of $\mathbb{Q}(\zeta_{p^2})/\mathbb{Q}$ in which $p$ is wildly ramified.
Then the Leopoldt conjecture holds for $F$ at the prime $p$.
\end{theorem}
As we already pointed out, the Leopoldt conjecture for abelian extensions of $\mathbb{Q}$ has been settled since \cite{MR220694}, over 50 years ago, so Theorem \ref{leopoldt conj special case} is not a new case of the Leopoldt conjecture at all. The noteworthy thing about the argument we have given, above, is its use of the $v_1$-periodicity $\pi_{-1}(L_{KU}S/p)\cong \pi_{2(p-1)-1}(L_{KU}S/p)\cong \pi_{4(p-1)-1}(L_{KU}S/p) \cong \pi_{6(p-1)-1}(L_{KU}S/p) \cong \dots $ in homotopy groups to deduce nonvanishing of the $p$-adic regulator. More generally:
\begin{observation} \label{how to try to prove leopoldt conj}
If $F$ is a totally real number field and if we have integers $j,k$ and spectra $E,X$ such that
\begin{enumerate}
\item for $n\gg 0$, the order of $\pi_{2(p^k-1)p^n - 1}(L_EX)$ is equal to the denominator of the rational number $\frac{\zeta_F(1-p^{n}(p^k-1))}{\zeta(1-p^{n}(p^k-1))}$, 
\item $X$ admits a self-map $\Sigma^{(2p^k-2)j}X \rightarrow X$ which induces an isomorphism in $E_*$-homology, and
\item $\pi_{-1}(L_EX)$ is finite,
\end{enumerate}
then the Leopoldt conjecture holds for $F$ at the prime $p$.

The argument is as follows: the sequence
\[ \#(\pi_{-1}(L_EX)), \#(\pi_{2pj(p^k-1)-1}(L_EX)), \#(\pi_{2p^2j(p^k-1)-1}(L_EX)), \dots\] 
is constant, so the order of vanishing of $\zeta_{F,p}(s)$ at $s=1$ is equal to the order of vanishing of $\zeta_{\mathbb{Q},p}(s)$ at $s=1$, so by the same argument using Colmez's $p$-adic class number formula as above, $\reg_{F,p}$ is nonzero.
\end{observation}

\begin{remark}\label{using observation on leopoldt}
Every $E(k-1)$-acyclic finite CW-complex $X$ admits a self-map $\Sigma^{(2p^k-2)j}X \rightarrow X$, for some $j$, which induces an isomorphism in $E(k)$-homology, by the periodicity theorem of Hopkins-Smith (see \cite{MR1652975}, or Theorem 1.5.4 of \cite{MR1192553}). Here $E(k)_*$ is the height $k$ $p$-primary Johnson-Wilson theory. So we have a very powerful mechanism for arranging for the second condition in Observation \ref{how to try to prove leopoldt conj} to be satisfied, and the third condition is, in many situations, amenable to explicit computation. It remains an open question how to produce spectra $X$ and {\em nonabelian} number fields $F$ satisfying the first condition in Observation \ref{how to try to prove leopoldt conj}, in order to prove potentially {\em new} cases of the Leopoldt conjecture. The Iwasawa-theoretic perspective sketched in Remark \ref{sorry everybody} represents my best hope for how one might go about producing such $X$ and $F$.
\end{remark}

\section{Appendix: a few entertaining numerical calculations.}
\label{computed examples appendix}
\subsection{Computed examples of values of $L(1-n,S/p)$.}

While Theorem~\ref{S/p realizability} describes the denominators of
$L(1-n,S/p)$ completely for positive integers $n$, it says 
nothing about the numerators. These numerators are much less predictable than the denominators. We include a table of a few values of $L(1-n,S/p)$ and the prime factorizations of their numerators, which might give the reader a sense of this unpredictability:
\begin{align*}
 L(0, S/3) &= 0 & \\
 L(-1, S/3) &= \frac{4}{3} \\ &= \frac{2^2}{3} \\
 L(-2, S/3) &= 0 & \\
 L(-3, S/3) &= \frac{796}{3} \\ &= \frac{2^2\cdot 199}{3} \\
 L(-4, S/3) &= 0 & \\
 L(-5, S/3) &= \frac{1409884}{3} \\ &= \frac{2^2\cdot 7 \cdot 43 \cdot 1171}{3} \\
 L(-6, S/3) &= 0 & \\
 L(-7, S/3) &= \frac{10595003836}{3} \\ &= \frac{2^2\cdot 2648750959}{3} 
\end{align*}\begin{align*}
 L(0, S/5) &= 0 & \\
 L(-1, S/5) &= 1136 \\ &= 2^4\cdot 71 \\
 L(-2, S/5) &= 0 & \\
 L(-3, S/5) &= \frac{607045659856}{5} \\ &= \frac{2^4\cdot 37940353741}{5} \\
 L(-4, S/5) &= 0 & \\
 L(-5, S/5) &= 1293561684322985119376 \\ &= 2^4\cdot 41^2 \cdot 3331 \cdot 2486381 \cdot 5807071 \\
 L(-6, S/5) &= 0 & \\
 L(-7, S/5) &= \frac{1280828318043498475058726863755856}{5} \\ &= \frac{2^4\cdot 401\cdot 1151 \cdot 1171 \cdot 281677007771\cdot 525827079851}{5} 
\end{align*}\begin{align*}
 L(0, S/7) &= 0 & \\
 L(-1, S/7) &= 17624384 \\ &= 2^6\cdot 113\cdot 2437 \\
 L(-2, S/7) &= 0 & \\
 L(-3, S/7) &= 60081275301219900531392 \\ &= 2^6 \cdot 547 \cdot 659 \cdot 7477 \cdot 348304469143 \\
 L(-4, S/7) &= 0 & \\
 L(-5, S/7) &= \frac{1448428968939581787932808098954336691322688}{7} \\ &= \frac{2^6\cdot 138054547\cdot 163933047708171216095114393777711}{7} \\
 L(-6, S/7) &= 0 & \\
 L(-7, S/7) &= 58235259522755629726600502123583976556247364608948281462604992 \\ &= 2^6 \cdot 14912003737\cdot 61019695682165635074111760577075533607839054420619 .
\end{align*}
These examples were computed by solving for Taylor coefficients to compute $B_n^{\chi}$ for $\chi\in \Dir(p^2)[p]$, as in Definition~\ref{def of gen berns}, and then multiplying the resulting values of $B_n^{\chi}$ as in the definition of $L(s,S/p)$ in Definition~\ref{def of S/p L-series}.
This process is not difficult to implement in a computer algebra package 
(the author did this in both MAGMA and SAGE).

So, for example, as a special case of Theorem~\ref{main thm summary}, we have
that the denominator of $L(-3,S/5)$ is the order of $\pi_{7}(L_{KU}S/5)$, i.e., the $5$ in the denominator of 
$L(-3, S/5) = \frac{607045659856}{5}$ is the numerical ``shadow'' of $\alpha_1\in \pi_7(S/5)$, while the $5$ in the denominator of $L(-7,S/5) = \frac{1280828318043498475058726863755856}{5}$ is the 
numerical ``shadow'' of $v_1\alpha_1 = \alpha_2 \in \pi_{15}(S/5)$.

\subsection{Some amusing probability arguments associated to homotopy groups.}

The functional equation of Corollary~\ref{S/p realizability of norms},
the Euler product of Proposition~\ref{euler product for L(S/p)}, and a computation of $L(-n,S/p)$, for a positive integer $n$, implies an asymptotic prime count.
The fact that the denominator of $L(-n,S/p)$ also counts the order of a homotopy group of $L_{KU}S/p$ tells us that the homotopy groups of $L_{KU}S/p$ have a relationship to the probability that certain ``randomly chosen'' collections of integers satisfy appropriate coprimality conditions.
These arguments are straightforward extensions of the classical interpretation of $1/\zeta(2)$ as the probability of two ``randomly chosen'' integers being coprime. 
For example:
\begin{question}\label{fun question}
Choose an odd prime $p$. Given a sequence $(m_1, n_1, m_2, n_2, \dots ,m_p,n_p)$ of ``randomly chosen'' integers, what is the probability that
the conditions:
\begin{enumerate}
\item
the integers $m_1, n_1, m_2, n_2, \dots ,m_p,n_p$ do not all share a common prime factor $\ell$ which is a primitive root modulo $p^2$, and
\item
none of the pairs $m_i,n_i$ share a common prime factor $\ell\neq p$ which is not a primitive root modulo $p^2$,
\end{enumerate}
both hold?
\end{question}
The answer to Question~\ref{fun question} is 
\[ \frac{1}{(1-p^{-2})\zeta(2)L(2,S/p)},\]
since $\frac{\ell^{2p}-1}{\ell^{2p}}$ is the probability that $m_1, n_1, \dots ,m_p,n_p$ are not all in the same residue class modulo $\ell$, 
and $\left(\frac{\ell^{2}-1}{\ell^{2}}\right)^p$ is the probability that, for each $i\in \{ 1, \dots ,p\}$, $m_i$ and $n_i$ are not in the same residue class modulo $\ell$.
Then the Euler product of Proposition~\ref{euler product for L(S/p)}
gives us
\[ 
\left(\prod_{\ell\notin G_p} \frac{\ell^{2}-1}{\ell^{2}}\right)^{p}\left(\prod_{\ell\in G_p} \frac{\ell^{2p}-1}{\ell^{2p}}\right) = \frac{1}{(1-p^{-2})\zeta(2)L(2,S/p)}.\]

We can use the functional equation in Corollary~\ref{S/p realizability of norms} to simplify $L(2,S/p)$. For example:
\begin{example}
Given a sequence $(m_1, n_1, m_2, n_2, m_3,n_3)$ of ``randomly chosen'' integers, 
let $P$ denote the probability that
the conditions:
\begin{enumerate}
\item
the integers $m_1, n_1, m_2, n_2, m_3,n_3$ do not all share a common prime factor $\ell$ which is a primitive root modulo $9$, and
\item
none of the pairs $m_i,n_i$ share a common prime factor $\ell$ which is not a primitive root modulo $9$,
\end{enumerate}
both hold.
Then
\begin{align} 
\nonumber P 
 &= \frac{1}{(1-3^{-2})}\frac{1}{\zeta(2)}\frac{1}{L(2,S/3)} \\
\nonumber  &= \frac{9}{8}\frac{6}{\pi^2} \left(\frac{ 3^{3}}{2\pi^2}\right)^{2} \frac{1}{L(-1,S/3)}\\
\label{little computation} &= \frac{3^9}{2^4\pi^6} \frac{3}{4}\\
\nonumber  &= \frac{59049}{64\pi^6},\end{align}
which is approximately a 96 percent chance. The factor of $3/4$ in~\eqref{little computation} is the reciprocal of $L(-1, S/3) = 4/3$, given above. By Theorem~\ref{S/p realizability}, the factor of three in the denominator of $L(-1,S/3)$, which is responsible for the probability $P$ being approximately 96 percent instead of approximately 32 percent, is the same factor which accounts for the 
nonvanishing of the fourth $KU$-local stable homotopy group $\pi_4(L_{KU}S/3)$ of the mod $3$ Moore spectrum.
\end{example}

\bibliography{/home/asalch/texmf/tex/salch}{}

\def\cprime{$'$} \def\cprime{$'$} \def\cprime{$'$} \def\cprime{$'$}
\begin{thebibliography}{10}

\bibitem{MR0198470}
J.~F. Adams.
\newblock On the groups {$J(X)$}. {IV}.
\newblock {\em Topology}, 5:21--71, 1966.

\bibitem{MR0434929}
Tom~M. Apostol.
\newblock {\em Introduction to analytic number theory}.
\newblock Springer-Verlag, New York-Heidelberg, 1976.
\newblock Undergraduate Texts in Mathematics.

\bibitem{bauer2011bousfield}
Tilman Bauer.
\newblock Bousfield localization and the {H}asse square, 2011.

\bibitem{MR220694}
Armand Brumer.
\newblock On the units of algebraic number fields.
\newblock {\em Mathematika}, 14:121--124, 1967.

\bibitem{MR0109132}
L.~Carlitz.
\newblock Arithmetic properties of generalized {B}ernoulli numbers.
\newblock {\em J. Reine Angew. Math.}, 202:174--182, 1959.

\bibitem{MR0104630}
L.~Carlitz.
\newblock Some arithmetic properties of generalized {B}ernoulli numbers.
\newblock {\em Bull. Amer. Math. Soc.}, 65:68--69, 1959.

\bibitem{MR922806}
Pierre Colmez.
\newblock R\'{e}sidu en {$s=1$} des fonctions z\^{e}ta {$p$}-adiques.
\newblock {\em Invent. Math.}, 91(2):371--389, 1988.

\bibitem{MR1333942}
Ethan~S. Devinatz and Michael~J. Hopkins.
\newblock The action of the {M}orava stabilizer group on the {L}ubin-{T}ate
  moduli space of lifts.
\newblock {\em Amer. J. Math.}, 117(3):669--710, 1995.

\bibitem{MR2030586}
Ethan~S. Devinatz and Michael~J. Hopkins.
\newblock Homotopy fixed point spectra for closed subgroups of the {M}orava
  stabilizer groups.
\newblock {\em Topology}, 43(1):1--47, 2004.

\bibitem{douglas2014topological}
Christopher~L Douglas, John Francis, Andr{\'e}~G Henriques, and Michael~A Hill.
\newblock {\em Topological modular forms}, volume 201.
\newblock American Mathematical Soc., 2014.

\bibitem{MR0392863}
Jean Fresnel.
\newblock Valeurs des fonctions z\^eta aux entiers n\'egatifs.
\newblock In {\em S\'eminaire de {T}h\'eorie des {N}ombres, 1970--1971 ({U}niv.
  {B}ordeaux {I}, {T}alence), {E}xp. {N}o. 27}, page~30. Lab. Th\'eorie des
  Nombres, Centre Nat. Recherche Sci., Talence, 1971.

\bibitem{MR1944041}
Philip~S. Hirschhorn.
\newblock {\em Model categories and their localizations}, volume~99 of {\em
  Mathematical Surveys and Monographs}.
\newblock American Mathematical Society, Providence, RI, 2003.

\bibitem{MR1652975}
Michael~J. Hopkins and Jeffrey~H. Smith.
\newblock Nilpotence and stable homotopy theory. {II}.
\newblock {\em Ann. of Math. (2)}, 148(1):1--49, 1998.

\bibitem{MR133304}
Helmut Klingen.
\newblock \"{U}ber die {W}erte der {D}edekindschen {Z}etafunktion.
\newblock {\em Math. Ann.}, 145:265--272, 1961/62.

\bibitem{MR163900}
Tomio Kubota and Heinrich-Wolfgang Leopoldt.
\newblock Eine {$p$}-adische {T}heorie der {Z}etawerte. {I}. {E}inf\"{u}hrung
  der {$p$}-adischen {D}irichletschen {$L$}-{F}unktionen.
\newblock {\em J. Reine Angew. Math.}, 214(215):328--339, 1964.

\bibitem{MR0092812}
Heinrich-Wolfgang Leopoldt.
\newblock Eine {V}erallgemeinerung der {B}ernoullischen {Z}ahlen.
\newblock {\em Abh. Math. Sem. Univ. Hamburg}, 22:131--140, 1958.

\bibitem{MR139602}
Heinrich-Wolfgang Leopoldt.
\newblock Zur {A}rithmetik in abelschen {Z}ahlk\"{o}rpern.
\newblock {\em J. Reine Angew. Math.}, 209:54--71, 1962.

\bibitem{MR0406981}
Stephen Lichtenbaum.
\newblock Values of zeta-functions, \'{e}tale cohomology, and algebraic
  {$K$}-theory.
\newblock In {\em Algebraic {$K$}-theory, {II}: ``{C}lassical'' algebraic
  {$K$}-theory and connections with arithmetic ({P}roc. {C}onf., {B}attelle
  {M}emorial {I}nst., {S}eattle, {W}ash., 1972)}, pages 489--501. Lecture Notes
  in Math., Vol. 342, 1973.

\bibitem{MR0172878}
Jonathan Lubin and John Tate.
\newblock Formal complex multiplication in local fields.
\newblock {\em Ann. of Math. (2)}, 81:380--387, 1965.

\bibitem{MR1697859}
J{\"u}rgen Neukirch.
\newblock {\em Algebraic number theory}, volume 322 of {\em Grundlehren der
  Mathematischen Wissenschaften [Fundamental Principles of Mathematical
  Sciences]}.
\newblock Springer-Verlag, Berlin, 1999.
\newblock Translated from the 1992 German original and with a note by Norbert
  Schappacher, With a foreword by G. Harder.

\bibitem{MR0315016}
Daniel Quillen.
\newblock On the cohomology and {$K$}-theory of the general linear groups over
  a finite field.
\newblock {\em Ann. of Math. (2)}, 96:552--586, 1972.

\bibitem{MR737778}
Douglas~C. Ravenel.
\newblock Localization with respect to certain periodic homology theories.
\newblock {\em Amer. J. Math.}, 106(2):351--414, 1984.

\bibitem{MR1192553}
Douglas~C. Ravenel.
\newblock {\em Nilpotence and periodicity in stable homotopy theory}, volume
  128 of {\em Annals of Mathematics Studies}.
\newblock Princeton University Press, Princeton, NJ, 1992.
\newblock Appendix C by Jeff Smith.

\bibitem{MR0404145}
Jean-Pierre Serre.
\newblock Formes modulaires et fonctions z\^{e}ta {$p$}-adiques.
\newblock In {\em Modular functions of one variable, {III} ({P}roc. {I}nternat.
  {S}ummer {S}chool, {U}niv. {A}ntwerp, 1972)}, pages 191--268. Lecture Notes
  in Math., Vol. 350, 1973.

\bibitem{MR1503335}
Carl~Ludwig Siegel.
\newblock \"{U}ber die analytische {T}heorie der quadratischen {F}ormen. {III}.
\newblock {\em Ann. of Math. (2)}, 38(1):212--291, 1937.

\bibitem{MR826102}
R.~W. Thomason.
\newblock Algebraic {$K$}-theory and \'etale cohomology.
\newblock {\em Ann. Sci. \'Ecole Norm. Sup. (4)}, 18(3):437--552, 1985.

\bibitem{MR1421575}
Lawrence~C. Washington.
\newblock {\em Introduction to cyclotomic fields}, volume~83 of {\em Graduate
  Texts in Mathematics}.
\newblock Springer-Verlag, New York, second edition, 1997.

\bibitem{MR1053488}
A.~Wiles.
\newblock The {I}wasawa conjecture for totally real fields.
\newblock {\em Ann. of Math. (2)}, 131(3):493--540, 1990.

\end{thebibliography}
\bibliographystyle{plain}
\end{document}